\documentclass[oneside,american,english,british]{amsart}
\usepackage[T1]{fontenc}
\usepackage{geometry}
\usepackage{amsthm}
\usepackage{amstext}
\usepackage{amssymb}
\usepackage{setspace}
\usepackage{units}
\usepackage{babel}
\usepackage[latin9,cp1255]{inputenc}
\PassOptionsToPackage{normalem}{ulem}
\usepackage{ulem}

\makeatletter
\numberwithin{equation}{section}
\numberwithin{figure}{section}

\theoremstyle{plain}
\newtheorem{thm}{\protect\theoremname}
  \theoremstyle{definition}
  \newtheorem{defn}[thm]{\protect\definitionname}
  \theoremstyle{definition}
  \newtheorem{problem}[thm]{\protect\problemname}
  \theoremstyle{plain}
  \newtheorem{prop}[thm]{\protect\propositionname}
  \theoremstyle{plain}
  \newtheorem{lem}[thm]{\protect\lemmaname}
  \theoremstyle{plain}
  \newtheorem{cor}[thm]{\protect\corollaryname}
  \theoremstyle{remark}
  \newtheorem{rem}[thm]{\protect\remarkname}
  \theoremstyle{definition}
  \newtheorem{example}[thm]{\protect\examplename}
  \theoremstyle{remark}
  \newtheorem{claim}[thm]{\protect\claimname}
  \theoremstyle{plain}
  \newtheorem{fact}[thm]{\protect\factname}
\numberwithin{thm}{section}

\makeatother
\usepackage{babel}
  \addto\captionsamerican{\renewcommand{\claimname}{Claim}}
  \addto\captionsamerican{\renewcommand{\corollaryname}{Corollary}}
  \addto\captionsamerican{\renewcommand{\definitionname}{Definition}}
  \addto\captionsamerican{\renewcommand{\examplename}{Example}}
  \addto\captionsamerican{\renewcommand{\factname}{Fact}}
  \addto\captionsamerican{\renewcommand{\lemmaname}{Lemma}}
  \addto\captionsamerican{\renewcommand{\problemname}{Problem}}
  \addto\captionsamerican{\renewcommand{\propositionname}{Proposition}}
  \addto\captionsamerican{\renewcommand{\remarkname}{Remark}}
  \addto\captionsamerican{\renewcommand{\theoremname}{Theorem}}
  \addto\captionsbritish{\renewcommand{\claimname}{Claim}}
  \addto\captionsbritish{\renewcommand{\corollaryname}{Corollary}}
  \addto\captionsbritish{\renewcommand{\definitionname}{Definition}}
  \addto\captionsbritish{\renewcommand{\examplename}{Example}}
  \addto\captionsbritish{\renewcommand{\factname}{Fact}}
  \addto\captionsbritish{\renewcommand{\lemmaname}{Lemma}}
  \addto\captionsbritish{\renewcommand{\problemname}{Problem}}
  \addto\captionsbritish{\renewcommand{\propositionname}{Proposition}}
  \addto\captionsbritish{\renewcommand{\remarkname}{Remark}}
  \addto\captionsbritish{\renewcommand{\theoremname}{Theorem}}
  \addto\captionsenglish{\renewcommand{\claimname}{Claim}}
  \addto\captionsenglish{\renewcommand{\corollaryname}{Corollary}}
  \addto\captionsenglish{\renewcommand{\definitionname}{Definition}}
  \addto\captionsenglish{\renewcommand{\examplename}{Example}}
  \addto\captionsenglish{\renewcommand{\factname}{Fact}}
  \addto\captionsenglish{\renewcommand{\lemmaname}{Lemma}}
  \addto\captionsenglish{\renewcommand{\problemname}{Problem}}
  \addto\captionsenglish{\renewcommand{\propositionname}{Proposition}}
  \addto\captionsenglish{\renewcommand{\remarkname}{Remark}}
  \addto\captionsenglish{\renewcommand{\theoremname}{Theorem}}
  \providecommand{\claimname}{Claim}
  \providecommand{\corollaryname}{Corollary}
  \providecommand{\definitionname}{Definition}
  \providecommand{\examplename}{Example}
  \providecommand{\factname}{Fact}
  \providecommand{\lemmaname}{Lemma}
  \providecommand{\problemname}{Problem}
  \providecommand{\propositionname}{Proposition}
  \providecommand{\remarkname}{Remark}
  \providecommand{\theoremname}{Theorem}
  
\begin{document}
\begin{onehalfspace}

\title{Perfect Set Theorems for Equivalence Relations with $I$-small
classes}

\author{Ohad Drucker}

\begin{abstract}
A classical theorem due to Mycielski states that an equivalence relation
$E$ having the Baire property and meager equivalence classes
must have a perfect set of pairwise inequivalent elements. We consider
equivalence relations with $I$-small equivalence classes, where
$I$ is a proper $\sigma$-ideal, and ask whether they have a perfect
set of pairwise inequivalent elements. We give a positive answer for
$E$ universally Baire. We show that the answer for $E$ $\mathbf{\Delta_{2}^{1}}$
is independent of $ZFC$, and find set theoretic assumptions equivalent to it when $I$ is the countable ideal.\

For equivalence relations which are $\mathbf{\Sigma^1_2}$ and with meager classes, we show that a perfect set of pairwise inequivalent elements exists whenever a Cohen real over $L[z]$ exists for any real $z$ -- which strengthens Mycielski's theorem.

A few comments are made about $\sigma$-ideals generated by $\Pi_{1}^{1}$
and orbit equivalence relations.
\end{abstract}

\maketitle
\section{Introduction}

We say that an equivalence relation $E$ on a Polish space $X$ \textit{has
perfectly many classes} if there is a perfect set $P\subseteq X$
such that all elements of $P$ are pairwise inequivalent.

Two classical theorems due to Mycielski claim:

\selectlanguage{american}%
\begin{thm}
\label{paper2_mycielski_meager}If $E$ is an equivalence relation that has
the Baire property, and all $E$-classes are meager, then $E$ has
perfectly many classes.
\end{thm}

\begin{thm}
\label{paper2_Mycielsky_measure_zero}If $E$ is an equivalence relation
that is Lebesgue measurable, and all $E$-classes are null, then
$E$ has perfectly many classes.
\end{thm}

This paper is about equivalence relations with small classes, and
investigate the cases in which such equivalence relations must have
many classes, namely, perfectly many classes. We will restrict our
discussion to equivalence relations which are not more complicated
than $\mathbf{\Sigma_{2}^{1}}$ or $\mathbf{\Pi_{2}^{1}}$. However,
we would like to consider a much wider class of notions of ``small''
sets:

\begin{defn}
Given a $\sigma$-ideal $I$ on a Polish space $X$, we say that
$A\subseteq X$ is \textit{$I$-positive} if $A\notin I$, and an
\textit{$I$-small set} if $A\in I$. We denote by $\mathbb{P}_{I}$
the partial order of Borel $I$-positive sets ordered by inclusion.
We say that $I$ is \textit{proper} if $\mathbb{P}_{I}$ is a proper
forcing notion.
\end{defn}

We can now state the main problem discussed in this paper:

\begin{problem}
\label{paper2_problem_perfect_set_property}Let $I$ be a proper $\sigma$-ideal and $E$ a $\mathbf{\mathbf{\Sigma_{\alpha}^{1},\ \Pi_{\alpha}^{1}}}$
or $\mathbf{\Delta_{\alpha}^{1}}$ equivalence relation with $I$-small classes. Does $E$ have perfectly many classes?

\subsection{The results of this paper}
To make statements easier, we fix the following notation:\end{problem}
\selectlanguage{english}%
\begin{defn}
\label{paper2_definition PSP}For $I$ a $\sigma$-ideal, $PSP_{I}(\mathbf{\Sigma_{n}^{1})}$
(for ``Perfect Set Property'') is the following statement:

``If $E$ is a $\mathbf{\Sigma_{n}^{1}}$ equivalence relation with
$I$-small classes then $E$ has perfectly many classes''.

\end{defn}

In section \ref{paper2_sec:UB} we prove the following:

\begin{thm}
Let $E$ be a universally Baire equivalence relation, and $I$ a proper
$\sigma$-ideal. If all $E$-classes are $I$-small, then $E$
has perfectly many classes.
\end{thm}

\begin{cor}
Let $E$ be an analytic equivalence relation, and $I$
a proper $\sigma$-ideal. If all $E$-classes are $I$-small,
then $E$ has perfectly many classes. In other words, for any proper $\sigma$-ideal
$I$, $PSP_{I}(\mathbf{\Sigma_{1}^{1})}$
is true.
\end{cor}

We could have stated the same for $E$ coanalytic, but that will follow immediately of Silver's theorem on coanalytic equivalence relations.
Note that some assumption on $I$ has to be made: given $E$ analytic
with uncountably many Borel classes but not perfectly many classes,
let $I_{E}$ be the $\sigma$-ideal generated by the equivalence
classes. Then all $E$ classes are $I_E$-small, but $E$ does not
have perfectly many classes. Indeed, such $I$ is never proper.

In section \ref{paper2_sec: Delta_1_2} we expand our discussion to the class
of $\mathbf{\Delta_{2}^{1}}$ equivalence relations. The case of provably
$\mathbf{\Delta_{2}^{1}}$ equivalence relations is no different then
the analytic case, since those are universally Baire.
But in the case of a general $\mathbf{\Delta_{2}^{1}}$ equivalence
relation, problem \ref{paper2_problem_perfect_set_property} is independent
of $ZFC$: 

\begin{thm}
Let $I$ be a proper $\sigma$-ideal, and assume $\Pi_{3}^{1}$-$\mathbb{P}_{I}$-generic absoluteness. Then $PSP_{I}(\mathbf{\Delta_{2}^{1})}$.
\end{thm}

\begin{thm}
\cite{paper2_chan}
If $\mathbb{R}=\mathbb{R}^{L[z]}$ for some $z\in\mathbb{R}$, then
for any $\sigma$-ideal $I,$ $\neg PSP_{I}(\mathbf{\Delta_{2}^{1})}$. 
\end{thm}

We use the above to completely solve the problem for the countable
ideal and $\mathbf{\Delta_{2}^{1}}$ equivalence relations:
\begin{thm}
(countable ideal) The following are equivalent:
\begin{enumerate}
\item $PSP_{countable}(\mathbf{\Delta_{2}^{1}).}$
\item For $z$ real, $\mathbb{R}^{L[z]}\neq\mathbb{R}.$
\end{enumerate}
\end{thm}

In section \ref{paper2_sec:sigma_1_2} we consider $\mathbf{\Sigma_{2}^{1}}$ and $\mathbf{\Pi_{2}^{1}}$
equivalence relations for the case of the meager ideal:
\begin{thm}\label{paper2_strong_myc}
If for any real $z$ there is a Cohen real over $L[z]$ then \[PSP_{meager}(\mathbf{\Sigma_{2}^{1})}\] and \[PSP_{meager}(\mathbf{\Pi_{2}^{1}}\ with\ Borel\ classes).\]
\end{thm}
That strengthens Mycielski's theorem \ref{paper2_mycielski_meager}: if
there are Cohen reals over any $L[z]$ but not comeager many,
$PSP_{meager}(\mathbf{\Sigma_{2}^{1})}$ is true although $\mathbf{\Sigma_{2}^{1}}$
sets do not necessarily have the Baire property, so one cannot use Mycielski's
theorem to prove so.

The last section elaborates on ideals generated by classes of a given
equivalence relation $E$ -- which we denote by $I_{E}$:
\begin{thm}
Let $E$ be a $\mathbf{\Pi_{1}^{1}}$ equivalence relation. Then $I=I_{E}$
is proper.
\end{thm}

\begin{thm}
If for every orbit equivalence relation $E$, $\mathbb{P}_{I_{E}}$
is proper, then the Vaught conjecture is true.
\end{thm}

\subsection{Borel Canonization of Analytic Equivalence Relations}

The following problem was raised by Kanovei, Sabok and Zapletal in
\cite{paper2_ksz}:

\selectlanguage{american}%
\begin{problem}
\label{paper2_Borel canonization e.r.}\textsl{Borel canonization of analytic
equivalence relations with Borel classes:} Given an analytic equivalence
relation $E$ on a Polish space $X$, all of its classes Borel, and
a proper $\sigma$-ideal $I$, does there exist an $I$-positive
Borel set $B$ such that $E$ restricted to $B$ is Borel?

That problem is strongly related to the main result of this paper
via the following celebrated theorem due to Silver:
\end{problem}

\begin{thm}
(Silver) Let $E$ be a coanalytic equivalence relation on a Polish
space $X$. Then either $E$ has countably many classes, or it has
perfectly many classes.
\end{thm}

Let $I$ be a proper $\sigma$-ideal, and let $E$ be an analytic
equivalence relation with Borel $I$-small classes. Assume a positive
answer to problem \ref{paper2_Borel canonization e.r.}, and fix $B$ a Borel
$I$-positive set such that $E\restriction_{B}$ is Borel. $B$
must intersect uncountably many classes, and Silver's theorem then
provides a perfect set of pairwise inequivalent elements. We have
thus proved the following:

\begin{prop}
A positive answer to problem \ref{paper2_Borel canonization e.r.} implies
that analytic equivalence relations with Borel $I$-small classes
for $I$ proper must have perfectly many classes.
\end{prop}

That was our original motivation to consider the problems discussed
in this paper. However, since the consequence of the positive answer
to problem \ref{paper2_Borel canonization e.r.} turned out to be a theorem
of $ZFC$, it hasn't shed new light on the problem of Borel canonization,
which is still open.

\subsection{Preliminaries}

The basics of universally Baire sets can be found in \cite{paper2_ub_paper}
or the relevant chapter in \cite{paper2_jech}. Forcing with ideals is thoroughly
covered in \cite{paper2_zapletal}. \cite{paper2_ikegami} contains all generic absoluteness
results used along the paper.

\subsection{Acknowledgments}

This research was carried out under the supervision of Menachem Magidor,
and would not be possible without his elegant ideas and deep insights.
The author would like to thank him for his dedicated help. The author
would also like to thank Marcin Sabok for hours of helpful discussions
and for introducing him with the problem of Borel canonization that
has naturally led to the problems discussed in this paper.

\section{\label{paper2_sec:UB}Universally Baire Equivalence Relations with $I$-small classes}

In the following section we prove:

\begin{thm}
\label{paper2_main result UB}Let $E$ be a universally Baire equivalence
relation, and $I$ a proper $\sigma$-ideal. If all $E$-classes
are $I$-small, then $E$ has perfectly many classes.
\end{thm}

\begin{cor}
\label{paper2_main result analytic coanalytic}Let $E$ be an analytic equivalence relation, and $I$ a proper $\sigma$-ideal.
If all $E$-classes are $I$-small, then $E$ has perfectly many
classes. In other words, for any proper $\sigma$-ideal $I$, $PSP_{I}(\mathbf{\Sigma_{1}^{1})}$ is true.
\end{cor}

\begin{rem}
The reader interested only in analytic equivalence
relations can avoid using the universally Baire definition of $E$
and rely on analytic absoluteness or Shoenfield's absoluteness instead.
For example, analytic equivalence relations remain
equivalence relations in all generic extensions because of Shoenfield's
absoluteness.
\end{rem}

We begin by describing an absoluteness property of universally Baire
equivalence relations which will play a central role in the proof
of theorem \ref{paper2_main result UB}:

\begin{prop}
\label{paper2_UB ER remains ER}Let $E$ be a universally Baire equivalence
relation. Then $E$ remains an equivalence relation in generic extensions
of the universe.
\end{prop}

\begin{proof}
For a forcing notion $\mathbb{P}$, fix trees $T,S\subseteq(\omega\times\omega\times\kappa)$
such that $E=p[T]$ and $^{\sim}E=p[S]$ in $\mathbb{P}$-generic
extensions of the universe. For $t\in\kappa^{<\omega},$ $(t)_{0}$
and $(t)_{1}$ denote 2 sequences of length $|t|$ given by some bijection
of $\kappa^{<\omega}$ and $(\kappa^{<\omega})^{2}.$ Similarly for
$(t)_{0},(t)_{1},(t)_{2}.$

We define trees $T_{r},T_{s},T_{t}$ whose well foundedness is equivalent
to reflexivity, symmetry and transitivity of $E$, respectively:
\[
(s,t)\in T_{r}\Leftrightarrow(s,s,t)\in S.
\]
\[
(s_{1},s_{2},t)\in T_{s}\Leftrightarrow\left((s_{1},s_{2},(t)_{0})\in T\right)\wedge\left((s_{2},s_{1},(t)_{1})\in S\right).
\]
\[
(s_{1},s_{2},s_{3},t)\in T_{t}\Leftrightarrow\left((s_{1},s_{2},(t)_{0})\in T\right)\wedge\left((s_{2},s_{3},(t)_{1})\in T\right)\wedge\left((s_{1},s_{3},(t)_{2})\in S\right).
\]

Absoluteness of well foundedness of trees concludes the proof.
\end{proof}

The following lemma is based on \cite{paper2_foreman_magidor}, theorem 3.4. We will say that \textit{$\mathbb{P}$ adds a new class} if $\mathbb{P}$ adds a real not equivalent to any ground model real:

\begin{lem}
\label{paper2_main_lemma-UB}Let $\mathbb{P}$ be a proper forcing notion, and $E$ a universally Baire equivalence relation.
If $\mathbb{P}$ adds a new class, then $E$ has perfectly many
classes.
\end{lem}

\begin{proof}
Consider the product $\mathbb{P}\times\mathbb{P}$, and let $\tau$
be a name for a real that is not equivalent to any ground model real. We denote by
$\tau_{l}$ and $\tau_{r}$ the ``left'' and ``right'' names of
that real, respectively. 

\begin{claim}
For every condition $p$, $(p,p)\nVdash\tau_{l}E\tau_{r}.$ 
\end{claim}

Given the claim, pick $\theta$ large enough and $M\preceq H_{\theta}$
a countable elementary submodel containing all the necessary information.
We construct a perfect tree $\langle p_{s}\ :\ s\in2^{<\omega}\rangle$
of conditions of $\mathbb{P}$ such that:

\begin{enumerate}
\item $p_{s^{\frown}i}\leq p_{s}$.
\item $p_{s}$ determines at least the first $|s|$ elements of $\tau$.
\item For $f\in2^{\omega}:$ $\langle p_{f\restriction_{n}}\ :\ n\in\omega\rangle$
generate a $\mathbb{P}$-generic filter over $M$.
\item For $f,g\in2^{\omega}:$ $\langle(p_{f\restriction_{n}},p_{g\restriction_{n}})\ :\ n\in\omega\rangle$
generate a $\mathbb{P}\times\mathbb{P}$-generic filter over $M$.
\item $(p_{s^{\frown}0},p_{s^{\frown}1})\Vdash\neg(\tau_{l}E\tau_{r}).$
\end{enumerate}

The construction is inductive. Fix $\langle D_{n}\ :\ n\in\omega\rangle$
an enumeration of the dense open subsets of $\mathbb{P}$ that belong
to $M$ , and $\langle D_{n}^{*}\ :\ n\in\omega\rangle$ an enumeration
of the dense open subsets of $\mathbb{P}\times\mathbb{P}$ that belong
to $M$. To construct the $(n+1)'th$ level of the tree, first extend
all $p_{s}$ of level $n$ to 
\[
(p_{s^{\frown}0},p_{s^{\frown}1})\Vdash\neg(\tau_{l}E\tau_{r}).
\]
Then extend all elements of the new level so that they will belong
to $D_{n}$, and extend all pairs of elements of the new level so
that they will belong to $D_{n}^{*}$. A final extension of the new
level will guarantee condition $(2)$ as well.

For $f\in2^{\omega}$, let $\tau_{f}$ be the realization of $\tau$
by the generic filter generated by $\langle p_{f\restriction_{n}}\ :\ n\in\omega\rangle$.
The function $f\to\tau_{f}$ is continuous, by $(2)$. Using $(5)$,
if $f\neq g$ and $s$ is such that $f\supseteq s^{\frown}0$ and
$g\supseteq s^{\frown}1$, then $(p_{s^{\frown}0,}p_{s^{\frown}1})$
is in the generic filter adding $\tau_{f}$ and $\tau_{g}$, and hence
\[
M[G_0][G_1]\models\neg(\tau_{f}E\tau_{g}).
\]
Since $E$ is universally Baire, $\mathbb{V}\models\neg(\tau_{f}E\tau_{g})$,
and $E$ has perfectly many classes.
\end{proof}

\begin{proof}
(of the claim) Assume otherwise, and let $p\in\mathbb{P}$ be such
that $(p,p)\Vdash\tau_{l}E\tau_{r}$ . Pick $\theta$ large enough
and $M\preceq H_{\theta}$ a countable elementary submodel containing
all the necessary information, and in particular $p\in M$. The idea will be to consider one $M$-generic filter which is in $\mathbb{V}$, and another filter which is generic over both $\mathbb{V}$ and $M$ -- where we use properness to guarantee its existence.

So first, let $q\leq p$ be $(M,p)$-generic. Let 
\[
p\in G_{0}\in\mathbb{V}
\]
 be a generic filter over $M$, and $q\in G_{1}$ a generic filter
over $\mathbb{V}$. Then $G_{1}$ is $M$-generic as well (to be
precise -- its intersection with $\mathbb{P}\cap M$ is $M$-generic),
and we may find 
\[
G_{2}\in\mathbb{V}[G_{1}]
\]
 such that $p\in G_{2}\subseteq\mathbb{P}\cap M$ and $G_{2}$ is
generic over both countable models $M[G_{0}]$ and $M[G_{1}]$. Then $G_{0}\times G_{2}$
and $G_{1}\times G_{2}$ are both generic over $M$ and contain $(p,p)$.
It follows that
\[
M[G_{0}][G_{2}]\models\tau_{G_{0}}E\tau_{G_{2}}
\]
\[
M[G_{1}][G_{2}]\models\tau_{G_{1}}E\tau_{G_{2}},
\]
and using the universally Baire definition: 
\[
\mathbb{V}[G_{1}]\models\left(\tau_{G_{0}}E\tau_{G_{2}}\right)\wedge\left(\tau_{G_{1}}E\tau_{G_{2}}\right).
\]
Since by proposition \ref{paper2_UB ER remains ER} $E$ is still an equivalence
relation in $\mathbb{V}[G_{1}]$, 
\[
\mathbb{V}[G_{1}]\models\tau_{G_{0}}E\tau_{G_{1}}.
\]
Since $\tau_{G_{0}}\in\mathbb{V}$, we conclude that $\tau_{G_{1}}$ does belong to a ground model
equivalence class, although $G_1$ is $\mathbb{V}$-generic and $\tau$ is a name of a new class. That is a contradiction.
\end{proof}

\begin{cor}
\label{paper2_corollary new class iff perfectly many classes}Let $\mathbb{P}$
be a proper forcing notion adding a real, and $E$ a universally Baire
equivalence relation. Then $\mathbb{P}$ adds a new class if and only
if $E$ has perfectly many classes. 
\end{cor}

\begin{proof}
One direction is the previous lemma. For the other, note that when
a new real is added to the universe, a new real is added to every
perfect set of the universe. If $B$ is a Borel set disjoint of some class $[z]_E$, it remains so in $\mathbb{P}$-generic extensions -- since for a universally Baire set, being empty is absolute between generic extensions. We will show that a perfect
set of pairwise $E$ inequivalent elements remains such in a $\mathbb{P}$-generic extension. Hence the new real in the perfect set $P$ has no choice but to belong to a new $E$-class.

Indeed, given $P$ a perfect tree of pairwise $E$ inequivalent elements,
there exists a tree $T_{P}$ whose well foundedness is equivalent
to the pairwise inequivalence of the branches of $P:$
\[
(s_{1},s_{2},t)\in T_{P}\Leftrightarrow\left((s_{1},s_{2})\in P\right)\wedge\left((s_{1},s_{2},(t)_{0})\in T\right)\wedge\left((s_{1},s_{2},(t)_{1})\in I\right)
\]
where $I$ is a tree such that $I_{xy}$ is well founded if and only
if $x=y$.
\end{proof}

\begin{proof}
(of theorem \ref{paper2_main result UB}) Assume otherwise -- $E$ does not
have perfectly many classes. Hence by lemma \ref{paper2_main_lemma-UB},
forcing with $\mathbb{P}_{I}$ does not add a new class. Fix $z\in\mathbb{V}$
and $B\in\mathbb{P}_{I}$ such that
\[
B\Vdash x_{gen}\in[z],
\]
where $x_{gen}$ stands for the generic real added by $\mathbb{P}_I$.
Let $M$ be an elementary submodel of the universe containing $z$
and all the relevant information. Let $x\in B$ be $M$-generic.
Then $M[x]\models xEz$, and using the universally Baire definition
of $[z]$ we know that $\mathbb{V}\models xEz.$ We have thus shown
that the $M$-generics in $B$ are all equivalent to $z$ -- and
in particular $[z]$ is $I$-positive, contradicting our assumption.
\end{proof}

\begin{rem}
Corollary \ref{paper2_corollary new class iff perfectly many classes} is
interesting in its own but not needed for the proof of theorem \ref{paper2_main result UB}.
\end{rem}

\section{\label{paper2_sec: Delta_1_2}$\mathbf{\Delta_{2}^{1}}$ Equivalence Relations
with $I$-small Classes}

In general, $\mathbf{\Delta_{2}^{1}}$ equivalence relations can have
$I$-small classes without having perfectly many classes:

\begin{thm}
\cite{paper2_chan}
\label{paper2_negative result delta_1_2}In $L$, there is a countable $\Delta_{2}^{1}$
equivalence relation that does not have perfectly many classes. 
\end{thm}

\begin{proof}
In $L$, consider the following equivalence relation: 
\[
xEy\Leftrightarrow\left(\forall\alpha\ admissible \ x\in L_{\alpha}\Leftrightarrow y\in L_{\alpha}\right).
\]
Since the constructibility rank of $x$ and the admissibility of ordinals are decided by a countable
model and by all countable models, $E$ is a $\Delta_{2}^{1}$ equivalence relation. All $E$
classes are countable, since all $L_{\alpha}'s$ are. We will show
that any perfect tree $T$ must have two equivalent elements.

Let $T\in L$ be perfect, and let $\alpha$ be such that $T \in L_{\alpha}$. Let $\beta$ be the first admissible ordinal greater then $\alpha$ such that $L_\beta$ has a real not in $L_\alpha$. Using \cite{paper2_chan} fact 9.5, $L_\alpha$ is countable in $L_\beta$. Since $T$ has uncountably many branches in $L_\beta$, there must be 
\[x \neq y\in [T]\cap L_\beta\]
that are not in $L_\alpha$. It follows that $x$ and $y$ are equivalent.
\end{proof}

\begin{cor}
If $\mathbb{R}=\mathbb{R}^{L[z]}$ for some $z\in\mathbb{R}$, then
for any $\sigma$-ideal $I,$ $\neg PSP_{I}(\mathbf{\Delta_{2}^{1})}$. 
\end{cor}

\begin{proof}
A relativization of the above argument.
\end{proof}

We turn now to the positive results involving $\mathbf{\Delta_{2}^{1}}$
equivalence relations.

A set $A$ is \textit{provably $\Delta_{2}^{1}$} if the equivalence of the
$\Sigma_{2}^{1}$ and the $\Pi_{2}^{1}$ definitions is a theorem
of $ZFC$, which is: there are a $\Sigma_{2}^{1}$ formula $\Phi(x)$
and a $\Pi_{2}^{1}$ formula $\Psi(x)$ such that \[ZFC\vdash\forall x:\ \Phi(x)\leftrightarrow\Psi(x)\]
and $\Phi$ is a definition of $A$. A set $A$ is \textit{provably $\mathbf{\Delta_{2}^{1}}$}
(boldface) if there is a parameter $z$ and formulas $\Phi(x,z),\ \Psi(x,z)$
which are $\Sigma_{2}^{1}$ and $\Pi_{2}^{1}$ , respectively, such
that all $ZFC$ models with the parameter $z$ satisfy
\[
\forall x:\ \Phi(x,z)\leftrightarrow\Psi(x,z).
\]
Note that the above formula is $\Pi_{3}^{1}(z).$

\begin{cor}
(of theorem \ref{paper2_main result UB}) Let $E$ be a provably $\mathbf{\Delta_{2}^{1}}$
equivalence relation, and $I$ a proper ideal. If all $E$-classes
are $I$-small, then $E$ has perfectly many classes. In other words,
$PSP_{I}(provably\ \mathbf{\Delta_{2}^{1})}$ for any proper $\sigma$-ideal $I$.
\end{cor}

\begin{proof}
It is easy to see that provably $\mathbf{\Delta_{2}^{1}}$ sets are
universally Baire. In fact, any set with a $\mathbf{\Delta_{2}^{1}}$
definition preserved in generic extensions is a universally Baire
set.
\end{proof}

Hence provably $\mathbf{\Delta_{2}^{1}}$ equivalence relations do
not present a new challenge. The rest of the section is devoted
to the case of a general $\mathbf{\Delta_{2}^{1}}$ equivalence relation.

We say that a forcing $\mathbb{P}$ has \textit{$\Pi_{3}^{1}$-$\mathbb{P}$-absoluteness} if $\mathbb{V}$ and $\mathbb{V}^{\mathbb{P}}$ agree
on $\Pi_{3}^{1}$ statements with parameters in $\mathbb{V}.$ For
most forcing notions $\mathbb{P},$ $\Pi_{3}^{1}$-$\mathbb{P}$-absoluteness is independent of $ZFC.$

\begin{thm}
\label{paper2_main_thm_Delta_1_2}Let $I$ be a proper $\sigma$-ideal,
and assume $\Pi_{3}^{1}$-$\mathbb{P}_{I}$-absoluteness.
Then $PSP_{I}(\mathbf{\Delta_{2}^{1})}$.
\end{thm}

The proof is a variant of the proof of theorem \ref{paper2_main result UB}.
We restate the lemmas and corollary in the new context and indicate
the main differences in the proofs.

\begin{proof}
Let $E$ be a $\mathbf{\Delta_{2}^{1}}$ equivalence relation with
$I$-small classes. We may assume $E$ is lightface $\Delta_{2}^{1}$.
Fix $\Phi(x,y)$ a $\Sigma_{2}^{1}$ formula and $\Psi(x,y)$ a $\Pi_{2}^{1}$
formula, both defining $E$, so that \[\mathbb{V}\models\forall x,y:\ \Phi(x,y)\Leftrightarrow\Psi(x,y).\]
Because of $\Pi_{3}^{1}$-$\mathbb{P}_{I}$-absoluteness, the
$\Sigma_{2}^{1}$ and $\Pi_{2}^{1}$ definitions will coincide in
all generic extensions of $\mathbb{V}.$ In particular, $E$ defined
by $\Phi$ and $\Psi$ will continue being an equivalence relation
in generic extensions -- using the above observations and Shoenfield's
absoluteness.

\begin{lem}
\label{paper2_main_lemma-for-Delta_1_2}Let $\mathbb{P}$ be a proper forcing
notion, and $E$ a $\mathbf{\Delta_{2}^{1}}$ equivalence
relation. Assume $\Pi_{3}^{1}$-$\mathbb{P}$-absoluteness.
Then if $\mathbb{P}$ adds a new $E$ class, then $E$ has perfectly
many classes.
\end{lem}

\begin{proof}
Consider the product $\mathbb{P}\times\mathbb{P}$, and let $\tau$, $\tau_{l}$ and $\tau_{r}$ be as in lemma \ref{paper2_main_lemma-UB}. $\Phi$ and $\Psi$ are as above. 

\begin{claim}
For every condition $p$, $(p,p)\nVdash\Phi(x,y)$, which in light
of the above is the same as $(p,p)\nVdash\Psi(x,y).$ 
\end{claim}

Given the claim, pick $\theta$ large enough and $M\preceq H_{\theta}$
a countable elementary submodel containing all the necessary information.
We construct a perfect tree $\langle p_{s}\ :\ s\in2^{<\omega}\rangle$
of conditions of $\mathbb{P}$ such that:

\begin{enumerate}
\item $p_{s^{\frown}i}\leq p_{s}$.
\item $p_{s}$ determines at least the first $|s|$ elements of $\tau$.
\item For $f\in2^{\omega}:$ $\langle p_{f\restriction_{n}}\ :\ n\in\omega\rangle$
generate a $\mathbb{P}$-generic filter over $M$.
\item For $f,g\in2^{\omega}:$ $\langle(p_{f\restriction_{n}},p_{g\restriction_{n}})\ :\ n\in\omega\rangle$
generate a $\mathbb{P}\times\mathbb{P}$-generic filter over $M$.
\item $(p_{s^{\frown}0},p_{s^{\frown}1})\Vdash\neg\Psi(\tau_{l},\tau_{r}).$
\end{enumerate}

From here we continue just as in the proof of lemma \ref{paper2_main_lemma-UB},
with analytic absoluteness enough to complete the proof.
\end{proof}

\begin{proof}
(of the claim) Exactly as in lemma \ref{paper2_main_lemma-UB}, with $xEy$
replaced by $\Phi(x,y)$, till the point we have 
\[
M[G_{0}][G_{2}]\models\Phi(\tau_{G_{0}},\tau_{G_{2}})
\]
\[
M[G_{1}][G_{2}]\models\Phi(\tau_{G_{1}},\tau_{G_{2}}).
\]
By analytic absoluteness: 
\[
\mathbb{V}[G_{1}]\models\Phi(\tau_{G_{0}},\tau_{G_{2}})\wedge\Phi(\tau_{G_{1}},\tau_{G_{2}}).
\]
As previously mentioned, $\Phi$ remains an equivalence relation in
$\mathbb{V}[G_{1}]$, and so 
\[
\mathbb{V}[G_{1}]\models\Phi(\tau_{G_{0}},\tau_{G_{1}}).
\]
Since $\tau_{G_{0}}\in\mathbb{V}$, we conclude that $\tau_{G_{1}}$ does belong to a ground model
equivalence class, although $G_1$ is $\mathbb{V}$-generic and $\tau$ is a name of a new class. That is a contradiction.
\end{proof}

Note that in the proof we have used both the $\Sigma_{2}^{1}$
and the $\Pi_{2}^{1}$ definitions.

\begin{cor}
\label{paper2_new_class_iff_perf_delta_1_2}
Let $\mathbb{P}$ be a proper forcing notion adding a real, and $E$
a $\mathbf{\Delta_{2}^{1}}$ equivalence relation. Assume $\Pi_{3}^{1}$-$\mathbb{P}$-absoluteness. Then $\mathbb{P}$ adds a new class
if and only if $E$ has perfectly many classes. 
\end{cor}

\begin{proof}
One direction is the above proof. The other is similar to the proof of corollary \ref{paper2_corollary new class iff perfectly many classes}, where absoluteness arguments now follow of Shoenfield's theorem.
\end{proof}

We can now complete the proof of theorem \ref{paper2_main_thm_Delta_1_2} in the same way we have proved theorem \ref{paper2_main result UB} -- here $M[x]\models xEz$ implies $\mathbb{V}\models xEz$ follows
of analytic absoluteness.
\end{proof}

Together with \cite{paper2_ikegami}, we have shown:

\begin{thm}
The following are equivalent:
\begin{enumerate}
\item $PSP_{countable}(\mathbf{\Delta_{2}^{1})}.$
\item For $z$ real, $\mathbb{R}^{L[z]}\neq\mathbb{R}.$
\end{enumerate}
\end{thm}

\begin{proof}
$(1)\Rightarrow(2)$ is theorem \ref{paper2_negative result delta_1_2}.
$(2)\Rightarrow(1)$ follows from theorem \ref{paper2_main_thm_Delta_1_2}
since Ikegami has shown in \cite{paper2_ikegami} that $(2)$ is equivalent
to $\Pi_{3}^{1}$-Sacks-absoluteness.
\end{proof}

\begin{rem}
For the case of the meager and null ideal, we have:

If for any $z$ real, there is a Cohen (random) real over $L[z]$,
then $PSP_{meager(null)}(\mathbf{\Delta_{2}^{1})}$ .

To see why, use theorem \ref{paper2_main_thm_Delta_1_2} together with the
fact that existence of Cohen (random) reals over any $L[z]$ is equivalent
to $\Pi_{3}^{1}$-Cohen (random) absoluteness.

However, this is not a new result -- it follows from Mycielski's theorems \ref{paper2_mycielski_meager} and \ref{paper2_Mycielsky_measure_zero} together with Ihoda-Shelah theorem on the Baire property (and Lebesgue
measurability) of $\Delta_{2}^{1}$ sets.
\end{rem}

\begin{rem}
In \cite{paper2_ikegami} theorem 4.3 it is proved that for a wide class of
$\sigma$-ideals, ''$\Pi_{3}^{1}$-$\mathbb{P}_{I}$-absoluteness'' is equivalent to ``all $\mathbf{\Delta_{2}^{1}}$
sets are $\mathbb{P}_{I}$-Baire''. A set is universally Baire
if and only if it is $\mathbb{P}$-Baire for every forcing notion
$\mathbb{P}.$ 

Using the above terminology and referring to ideals to which \cite{paper2_ikegami}
theorem 4.3 applies, a result of section \ref{paper2_sec:UB} is that if
every $\mathbf{\Delta_{2}^{1}}$ set is $\mathbb{P}$-Baire for any
$\mathbb{P}$, and $I$ is any proper ideal, then $PSP_{I}(\mathbf{\Delta_{2}^{1})}$.
Section \ref{paper2_sec: Delta_1_2} shows that if for a given proper ideal
$I$, every $\mathbf{\Delta_{2}^{1}}$ set is $\mathbb{P}_{I}$-Baire, then $PSP_{I}(\mathbf{\Delta_{2}^{1})}$. In that sense, section
\ref{paper2_sec: Delta_1_2} provides a ''local'' version of the result of
section \ref{paper2_sec:UB}.
\end{rem}

\section{\label{paper2_sec:sigma_1_2}$\mathbf{\Sigma_{2}^{1}}$ and $\mathbf{\Pi_{2}^{1}}$ Equivalence Relations
with Meager Classes}

In this section we focus our attention on the meager ideal.

Note that until now, we have not given any new result on equivalence relations with meager classes. Considering section \ref{paper2_sec:UB}, for example, if $E$ is universally Baire with meager classes, then it has the Baire property, and then Mycielski's theorem \ref{paper2_mycielski_meager} is valid. Regarding section \ref{paper2_sec: Delta_1_2}, whenever forcing with non-meager Borel sets has $\Pi^1_3$ generic absoluteness, then $\mathbf{\Delta^1_2}$ sets have the Baire property -- and yet again theorem \ref{paper2_mycielski_meager} applies. The following section introduces a case in which Mycielski's theorem does not apply and we can still obtain the desired perfect set property for equivalence relations with meager classes.

For the following recall that the existence of Cohen reals over $L[z]$ for any real $z$ is equivalent to $\Pi_{3}^{1}$-Cohen absoluteness. Unless otherwise noted, $I = meager$.

\begin{thm}
\label{paper2_main_thm_Sigma_1_2}If for any real $z$ there is a Cohen real over $L[z]$ then \[PSP_{meager}(\mathbf{\Sigma_{2}^{1})}\] and \[PSP_{meager}(\mathbf{\Pi_{2}^{1}}\ with\ Borel\ classes).\]
\end{thm}

\begin{lem}
\label{paper2_main lemma for sigma^1_2} Assume that for any real $z$ there is a Cohen real over $L[z]$. Let $E$ be a $\mathbf{\Sigma_{2}^{1}}$ equivalence relation or a $\mathbf{\Pi_{2}^{1}}$ equivalence
relation. If the $\mathbb{P}_{I}$-generic belongs to a new $E$-class then
$E$ has perfectly many classes.
\end{lem}

\begin{proof}
The $\Pi_{3}^{1}$-$\mathbb{P}_{I}$-absoluteness guarantees
that $E$ will remain an equivalence relation in $\mathbb{P}_{I}$-generic extensions.

For ease of notation, we assume $E$ is lightface $\Sigma_{2}^{1}$
or $\Pi_{2}^{1}$ . Consider the product $\mathbb{P}_{I}\times\mathbb{P}_{I}$,
and let $\tau$ be a name for the $\mathbb{P}_I$-generic, $\tau_{l}$ a name for the $\mathbb{P}_I$-generic added by the left $\mathbb{P}_I$ and $\tau_{r}$ a name for the generic added by the right one.

\begin{claim}
For every condition $p$, $(p,p)\nVdash(\tau_{l}E\tau_{r})$.
\end{claim}

Assume the claim. When $E$ is $\Pi_{2}^{1}$, the proof continues in exactly the same way it did in the previous section. For $E\ $ $\Sigma_{2}^{1}$, we will construct a perfect tree $\langle p_{s}\ :\ s\in2^{<\omega}\rangle$
of elements of $\mathbb{P}_{I}$ such that:

\begin{enumerate}
\item $p_{s^{\frown}i}\leq p_{s}$.
\item $p_{s}$ determines at least the first $|s|$ elements of $\tau$.
\item For $f\in2^{\omega}:$ $\langle p_{f\restriction_{n}}\ :\ n\in\omega\rangle$
generate a $\mathbb{P}_{I}$-generic filter \textbf{over $L$}.
\item For $f,g\in2^{\omega}:$ $\langle(p_{f\restriction_{n}},p_{g\restriction_{n}})\ :\ n\in\omega\rangle$
generate a $\mathbb{P}_{I}\times\mathbb{P}_{I}$-generic filter
\textbf{over $L$}.
\item $L\models(p_{s^{\frown}0},p_{s^{\frown}1})\Vdash\neg(\tau_{l}E\tau_{r})$, which in our case of Cohen forcing is just the same as \[(p_{s^{\frown}0},p_{s^{\frown}1})\Vdash\neg(\tau_{l}E\tau_{r}).\]
\end{enumerate}

For the construction we rely on the following fact:

\begin{fact}
(\cite{paper2_brendle} 1.1) If there is a Cohen real over $L[z]$ then there is a perfect set
of $\mathbb{P}_{I}\times\mathbb{P}_{I}$ generics over $L[z]$.
\end{fact}

All we need to do now is to refine the perfect tree of the $\mathbb{P}_{I}\times\mathbb{P}_{I}$
generics. Shoenfield's absoluteness completes the
proof: if $L[x][y]\models\neg(\tau_{l}E\tau_{r})$ then $\mathbb{V}\models\neg(\tau_{l}E\tau_{r}).$
\end{proof}

\begin{proof}
(of the claim) If $E$ is $\Sigma_{2}^{1}$, the proof of the previous section works. We give the proof for $E\ \Pi_{2}^{1}$. The fact that a Cohen generic over $
\mathbb{V}$ is generic over all inner models of $\mathbb{V}$ is used over and over again.

Assume the claim fails, and let $p\in\mathbb{P}_I$ be such
that $(p,p)\Vdash\tau_{l}E\tau_{r}$ . Let 
\[
p\in G_{0}\in\mathbb{V}
\]
 be a generic filter over $L$ -- there is one, since when a Cohen real over $L$ exists, every non meager set has one. Let $p\in G_{1}$ be a generic filter
over $\mathbb{V}$, 
and let $G_{2}$ be generic over $\mathbb{V}[G_1]$ such that $p \in G_2$. Then $G_{0}\times G_{2}$
and $G_{1}\times G_{2}$ are both generic over $L$ and contain $(p,p)$.
It follows that
\[
L[G_{0}][G_{2}]\models\tau_{G_{0}}E\tau_{G_{2}}
\]
\[
L[G_{1}][G_{2}]\models\tau_{G_{1}}E\tau_{G_{2}}.
\]
By Shoenfield's absoluteness, these statements are still true
in $\mathbb{V}[G_{1}][G_{2}]$ . Recall that $\mathbb{P}_I$ and $\mathbb{P}_I \times \mathbb{P}_I$ are equivalent, therefore $\Pi^1_3$ absoluteness still applies for $\mathbb{P}_I \times \mathbb{P}_I$ and $E$ is transitive in $\mathbb{V}[G_{1}][G_{2}]$. 
Using absoluteness again we see that \[\mathbb{V}[G_{1}]\models\tau_{G_{0}}E\tau_{G_{1}}.\] But $\tau_{G_{0}}\in\mathbb{V}$,
whereas $\tau_{G_{1}}$ is generic over $\mathbb{V}$ , so $\tau_{G_{1}}$
belongs to a ground model equivalence class -- which is a contradiction.
\end{proof}

\begin{proof}
(of theorem \ref{paper2_main_thm_Sigma_1_2}) For $E$ $\mathbf{\Sigma_{2}^{1}}$, exactly as in the previous section. For $E$ $\mathbf{\Pi_{2}^{1}}$, one uses the additional assumption that the classes are Borel, in which case the $\mathbb{P}_{I}$-generic must belong to a new $E$-class.
\end{proof}

Theorem \ref{paper2_main_thm_Sigma_1_2} is indeed stronger than Mycielski's theorem \ref{paper2_mycielski_meager} in the following sense -- in a universe in which there are Cohen reals over any $L[z]$
but not comeager many, $PSP_{meager}(\mathbf{\Sigma_{2}^{1})}$ is
true but $\mathbf{\Sigma_{2}^{1}}$ sets do not necessarily have the Baire property.
\begin{rem} We conjecture that $PSP_{meager}({\mathbf{\Sigma^1_2})}$ is equivalent to the existence of Cohen generics over $L[z]$ for any real $z$.
\end{rem}

\section{$\sigma$-ideals generated by equivalence relations}

Given an equivalence relation $E$ , let $I_{E}$ be the $\sigma$-ideal generated by the $E$-equivalence classes.

\begin{example}
For $x,y\in\omega^{\omega}$, let 
\[
xE_{ck}y\Leftrightarrow\omega_{1}^{ck(x)}=\omega_{1}^{ck(y)}.
\]
 Let $x_{gen}$ be the generic real added by forcing with $\mathbb{P}_{I_{E_{ck}}}$.
Then $\omega_{1}^{ck(x_{gen})}\geq\omega_{1}$, and in particular, $I_{E_{ck}}$
is improper.
\end{example}

\begin{example}
Assume the Vaught conjecture is false, and let $(G,X)$ be a counterexample
($G$ a Polish group and $X$ a Polish space). Let $E=E_{G}^{X}$
be the induced equivalence relation, and $\delta$ a Hjorth rank associated
with the action (see \cite{paper2_hjorth,paper2_my_paper_hjorth}). Recall that for a countable ordinal $\alpha$, 
\[
\mathcal{A}_{\alpha}=\{x\ :\ \delta(x)\leq\alpha\}
\]
 is Borel and the orbit equivalence relation restricted to $\mathcal{A}_{\alpha}$
is Borel as well. Silver's theorem now guarantees that $\mathcal{A}_{\alpha}$
is a \textbf{countable} union of equivalence classes -- therefore
$\mathcal{A}_{\alpha}\in I_{E}.$ The generic real $x_{gen}$ added
by $\mathbb{P}_{I_{E}}$ must then have rank at least $\omega_{1}$,
proving the improperness of $I_{E}.$
\end{example}

\begin{thm}
\label{paper2_properness_of_I_E}Let $E$ be an analytic or coanalytic equivalence
relation such that every Borel set intersecting uncountably many classes,
has perfectly many classes. Then $I=I_{E}$ is proper.
\end{thm}

\begin{proof}
Pick $\theta$ large enough and $M\preceq H_{\theta}$ a countable
elementary submodel, and let $B\in M$ be a Borel $I$-positive
set. We will find a perfect set of pairwise inequivalent elements,
all in $B$ and generic over $M$ -- therefore proving the properness
of $I$ .

Consider the product $\mathbb{P}_{I}\times\mathbb{P}_{I}$, and let
$\tau$ be a name for the generic real. We denote by $\tau_{l}$ and
$\tau_{r}$ the ``left'' and ``right'' names of the new real,
respectively. 

\begin{claim}
For every condition $B$, $(B,B)\nVdash\tau_{l}E\tau_{r}.$ 
\end{claim}

\begin{proof}
Let $B\in\mathbb{P}_{I}$. Then $B$ intersects uncountably many classes,
hence by the assumption it contains a perfect set of pairwise inequivalent
elements. It is easy to see that $B$ contains two disjoint
perfect sets $B_0$ and $B_1$, both of which of pairwise inequivalent elements, such that their saturations are disjoint.
If
\[
B_{0}\times B_{1}\Vdash\neg(\tau_{l}E\tau_{r}).
\]
the proof of the claim will be completed. Indeed, 
\[
\mathbb{V}\models\forall x\in B_{0}\ \forall y\in B_{1}\ \neg(xEy),
\]
which is a $\mathbf{\Pi_{2}^{1}}$ statement, therefore $\mathbb{V}[G_{0}][G_{1}]\models\neg(\tau_{l}E\tau_{r})$.
\end{proof}

We can now fix $M \preceq H_\theta$ a countable elementary submodel and repeat the same construction carried out in the
proof of lemma \ref{paper2_main_lemma-UB}, resulting in a perfect tree
of conditions. The different
branches through the tree induce a perfect set $P$ of mutually $M$-generic elements. For $x\neq y$ in $P$,
\[
M[x][y]\models\neg(xEy)
\]
and absoluteness completes the proof.
\end{proof}

\begin{cor}
Let $E$ be a $\mathbf{\Pi_{1}^{1}}$ equivalence relation. Then $I=I_{E}$
is proper.
\end{cor}

\begin{proof}
By Silver's theorem, every coanalytic equivalence relation satisfies
the condition of theorem \ref{paper2_properness_of_I_E}.
\end{proof}

\begin{cor}
Let $E$ be an analytic equivalence relation, and $I=I_{E}$. Then
$\mathbb{P}_{I}$ is proper if and only if every Borel set intersecting
uncountably many classes, has perfectly many classes. In particular,
if for every orbit equivalence relation $E$, $\mathbb{P}_{I_{E}}$
is proper, then the Vaught conjecture is true.
\end{cor}

\begin{proof}
One direction is corollary \ref{paper2_main result analytic coanalytic} restricted to a Borel $I$-positive set. The other is theorem \ref{paper2_properness_of_I_E}.
\end{proof}

\selectlanguage{american}%

\end{onehalfspace}
\end{document}